\documentclass[oneside]{amsart}
\usepackage{amsmath,amsthm,amssymb}
\usepackage[pdftex, plainpages=false]{hyperref}

\numberwithin{equation}{section}
\newtheorem{main}{Theorem}

\newtheorem{lemma}{Lemma}[section]
\newtheorem{theorem}[lemma]{Theorem}
\newtheorem{prop}[lemma]{Proposition}
\newtheorem{cor}[lemma]{Corollary}

\theoremstyle{definition}

\newtheorem{example}[lemma]{Example}

\theoremstyle{remark}
\newtheorem{remark}[lemma]{Remark}
\newtheorem{question}[lemma]{Question}
\newcommand{\R}{\mathbb R}
\newcommand{\Z}{\mathbb Z}
\newcommand{\co}{\colon}
\newcommand{\pd}{\partial}
\newcommand{\ep}{\varepsilon}

\def\g{\gamma}
\newcommand\diam{\operatorname{diam}}
\newcommand\Aut{\operatorname{Aut}}
\newcommand\Out{\operatorname{Out}}
\newcommand\Inn{\operatorname{Inn}}
\newcommand\Diff{\operatorname{Diff}}
\newcommand\Mod{\operatorname{Mod}}
\def\be{\begin{equation}}
\def\ee{\end{equation}}

\title{Manifolds without conjugate points and their fundamental groups}

\begin{document}



\thanks{The first author was supported in part by RFBR grant 11-01-00302-a.
The second author was supported in part by a Discovery grant from NSERC}

\author{Sergei Ivanov}
\address{Sergei Ivanov\\St.~Petersburg department of Steklov Math Institute RAS\\
191023, Fontanka 27, Saint Petersburg, Russia}
\email{svivanov@pdmi.ras.ru}

\author{Vitali Kapovitch}
\address{Vitali Kapovitch\\Department of Mathematics\\University of Toronto\\
Toronto, Ontario, M5S 2E4, Canada}\email{vtk@math.toronto.edu}

\begin{abstract}
We show that in the fundamental groups of  closed manifolds without conjugate points centralizers of all elements virtually split.
\end{abstract}

\maketitle \thispagestyle{empty}

\section{Introduction}
It is a classical consequence of Rauch comparison that manifolds of nonpositive curvature have no conjugate points.
While the converse need not hold even for closed manifolds~\cite{Gull75},  the following question
remains unanswered in dimensions above 2.

\begin{question}\label{q-nonpos}
Does every closed Riemannian manifold without conjugate points admit a metric of nonpositive curvature?
\end{question}

While the answer to this is likely negative, it has been shown that fundamental groups of
closed manifolds without conjugate points satisfy many properties
which are known to hold for nonpositively curved manifolds.
In particular, Croke and Schroeder \cite{CS} proved that 
if a closed manifold $\bar M$ admits
an \emph{analytic} metric without conjugate points then
every abelian subgroup of $\pi_1(\bar M)$ is straight
(i.e.\ quasi-isometrically embedded in $\pi_1(\bar M)$)
and every solvable subgroup of $\pi_1(\bar M)$  is virtually abelian.
Both of these properties are known to hold for groups which act isometrically,  properly  and co-compactly on CAT(0)-spaces
(we refer to such groups as CAT(0)-groups) or, more generally, for the  semi-hyperbolic groups
(which is a strictly  bigger class of groups,
see e.g. ~\cite{AB95} or ~\cite{HB99} for the definition).
Our main result is that the following known property of CAT(0)-groups
also holds for fundamental groups of closed manifolds without conjugate points. 

\begin{main}\label{t-main}
Let $\bar M$ be a closed manifold that admits
a $C^\infty$ Riemannian metric without conjugate points.
Then for every nontrivial element $\gamma\in\pi_1(\bar M)$,
its centralizer $Z(\gamma)<\pi_1(\bar M)$ virtually splits
over $\gamma$.

This means that there exists a finite index subgroup
$G<Z(\gamma)$ which is isomorphic to a direct product $\Z\times G'$
so that $\gamma$ corresponds to the generator of the $\Z$ factor.
\end{main}

Actually, we prove an equivalent but more convenient property
of the centralizer (Corollary~\ref{c-infinite-order}):
the image of $\gamma$ in $H_1(Z(\gamma))=Z(\gamma)/[Z(\gamma),Z(\gamma)]$
is a non-torsion element.

The proof builds upon the already mentioned work of Croke and Schroeder \cite{CS}.
The main result of \cite{CS} (i.e.\ the straightness of abelian subgroups)
by itself does not imply the assertion of Theorem~\ref{t-main}
but some intermediate results in \cite{CS} do.
We remove the analyticity assumption from one of those results
(see Proposition~\ref{p-const-difference} below)
and deduce Theorem~\ref{t-main} from it.

\begin{remark}
Kleiner (unpublished) and Lebedeva \cite{Lebedeva}
found simpler proofs of the main result of \cite{CS}
that work without analyticity.
These proofs go via a different route that do not
yield Theorem~\ref{t-main}.
\end{remark}

\begin{remark}
Our proof requires $C^\infty$ regularity of the metric or,
more precisely, $C^r$ regularity for some $r$
depending on $n=\dim M$.
This is needed in Lemma~\ref{l-analytic}.
We do not know whether Theorem~\ref{t-main} is true for
$C^r$ metrics for any fixed $r$.

Examples constructed in \cite{BIK} show that
Lemma~\ref{l-analytic} and an analogue of
Proposition~\ref{p-const-difference}
(in a similar but not identical context)
fail if $n\gg r$.
However it might be possible that some
features of metrics without conjugate points
(e.g.\ the fact that displacement functions
do not have critical points other than minima)
can be used to squeeze more regularity out
of the problem.
\end{remark}

\begin{remark}
A straightforward modification of the proof
shows that Theorem~\ref{t-main} holds for
Finsler metrics as well,
but we do not bother the reader with this generalization.
\end{remark}

\subsection*{Organization of the paper}
In Section~\ref{s:prelim} we provide necessary background material on the properties of
universal covers of closed manifolds without conjugate points.
In Section~\ref{s:busemann} we prove the key technical result Proposition~\ref{p-const-difference}
which allows us to remove the analyticity assumption from the arguments in~\cite{CS}.
In Section~\ref{sec-proof-main} we prove Theorem~\ref{t-main}.
In Section~\ref{s:examples} we construct new examples of manifolds that can be shown
not to admit metrics without conjugate points by Theorem~\ref{t-main} but not by previously known results.
In Section~\ref{s:solv} we prove that fundamental groups of closed manifolds without conjugate points
have solvable word and conjugacy problems (Theorem~\ref{t-solv-word}).
In Section~\ref{section-no-focal} we consider an analogue of Question~\ref{q-nonpos}
for metrics without \emph{focal} points and prove that the answer is affirmative
in dimension~3 (Theorem~\ref{3-no-focal}).

Lastly, in Section~\ref{s:open} we discuss a number of open problems concerning manifolds without conjugate points.

\smallskip

{\it Acknowledgements.}
The first author is grateful to the organizers of 
``Geometry in Inverse Problems'' program held in March--April 2012
at Fields Institute, Toronto, where this work started.
The second author is very grateful to Misha Kapovich and Ilya Kapovich
for many helpful conversations and suggestions during the preparation of this paper.

\section{Preliminaries}\label{s:prelim}
In this section we collect the necessary preliminaries,
borrowed mainly from \cite{CS}.

Let $\bar M$ be a compact Riemannian manifold without conjugate points
and $M$ its universal cover.
We represent the fundamental group $\pi_1(\bar M)$ as the group $\Gamma$
of deck transformations of $M$. This group
acts on $M$ by isometries and $\bar M=M/\Gamma$.
We fix the notation $\bar M$, $M$ and $\Gamma$ for the rest of the paper.

Since the metric has no conjugate points,
every geodesic in $M$ is minimizing,
$\exp_x\co T_xM\to M$ is a diffeomorphism
for every $x\in M$, and the Riemannian distance
function $d\co M\times M\to\R_+$ is smooth
outside the diagonal.
All geodesics throughout the paper are
parametrized by arc length.

Fix a nontrivial element $\gamma\in\Gamma$.
The \textit{displacement function} $d_\gamma\co M\to\R_+$
is defined by
$$
 d_\gamma(x) = d(x,\gamma x), \qquad x\in M.
$$
A complete geodesic $c\co\R\to M$ is called  an \textit{axis} of $\gamma$
if $\gamma$ translates $c$ forward along itself, i.e., there is a constant
$L>0$ such that $\gamma c(t)=c(t+L)$ for all $t\in\R$.
Note that if $c$ is an axis of $\gamma$, then the reverse geodesic
$t\mapsto c(-t)$ is an axis of $\gamma^{-1}$.
Let $A_\gamma\subset M$ denote the union of all axes of $\gamma$.

The following lemma summarizes several results from~\cite{CS} that we will need in what follows.

\begin{lemma}\label{l-axes-trivia}
In the above notation, the following holds.
\begin{enumerate}
\item\label{i-mind}
The function $d_\gamma$ assumes a positive minimum, $\min d_\gamma$.
The set of points $x\in M$ where $d_\gamma(x)=\min d_\gamma$,
is equal to $A_\gamma$.
\item\label{i-shift}
The isometry $\gamma$ translates all its axes by the same amount,
namely $\min d_\gamma$. That is, if $c$ is an axis of $\gamma$ then
$\gamma(c(t))=c(t+\min d_\gamma)$ for all $t\in\R$.
\item\label{i-mindm}
$\min d_{\gamma^m}=m\cdot\min d_\gamma$ for every integer $m\ge 1$.

\item\label{l-axes-crit}
 $A_\gamma$ is equal to the set of critical points of $d_\gamma$.
 In particular $d_\gamma$ has no critical points outside its minimum set.
\end{enumerate}
\end{lemma}

\begin{proof}
See \cite[Lemma 2.1]{CS} and remarks there.
\end{proof}

The following lemma is proved in \cite[Lemma 2.4]{CS},
however it is not made very clear there that the proof
does not depend on analyticity of the metric.
Therefore we include a proof (essentially the same one) here.

\begin{lemma}\label{l-connected}
The set $A_\gamma$ is connected. 
\end{lemma}

\begin{proof}
Consider the centralizer $Z(\gamma)\subset\Gamma$ of $\gamma$.
It is easy to see that the action of $Z(\gamma)$ on $M$
preserves $d_\gamma$, hence  $d_\gamma$ induces a well defined
function $\bar d_\gamma$ on $M/Z(\gamma)$.
Furthermore $\bar d_\gamma$ is proper
(see \cite[Lemma 2.2]{CS})
and by Lemma~\ref{l-axes-trivia}(\ref{l-axes-crit}) it has no critical point
outsize its minimum set $\bar A_\gamma$.
Therefore by Morse theory every sublevel set
$$
\bar U_\ep=\{x\in M/Z(\gamma):\bar d_\gamma(x)\le \min\bar d_\gamma+\ep\},
\qquad\ep>0,
$$
is a strong deformation retract of $M/Z(\gamma)$.
Since the projection $\pi\co M\to M/Z(\gamma)$ is a covering map,
it follows that the set
$$
 U_\ep:=\pi^{-1}(\bar U_\ep) = \{x\in M:d_\gamma(x)\le \min d_\gamma+\ep\}
$$
is a strong deformation retract of $M$.
Hence $U_\ep$ is connected. Since $A_\gamma = \bigcap_{\ep>0} U_\ep$,
it follows that $A_\gamma$ is connected as well.
\end{proof}

\begin{remark}\label{centralizer-fg}
The proof of Lemma~\ref{l-connected}  implies that $Z(\gamma)$ is finitely generated and finitely presented
since it shows that $M/Z(\gamma)$ is homotopy equivalent to $\bar U_\ep$ which is
compact manifold with boundary.
\end{remark}

The \textit{Busemann function} of a (minimizing)
geodesic $c\co\R\to M$ is a function $b_c\co M\to\R$
defined by
$$
 b_c(x) = \lim_{t\to+\infty} d(x,c(t))-t .
$$
The triangle inequality implies that the function
$t\mapsto d(x,c(t))-t$ is non-increasing,
hence $b_c(x)$ is well defined and $b_c(x)\le d(x,c(t))-t$
for all $t\in\R$. Note that $b_c(c(t))=-t$ for all $t\in\R$.

Clearly Busemann functions are 1-Lipschitz.
Busemann functions are naturally translated by isometries,
namely if $\alpha\co M\to M$ is an isometry,
then $b_c(x)=b_{\alpha c}(\alpha x)$ for all $x\in M$.
Changing the origin of a geodesic adds a constant to its
Busemann function, namely
if $c_1(t)=c(t+L)$ where $L$ is a constant,
then $b_{c_1}(x)=b_c(x)+L$.

\begin{lemma}\label{l-busemann-trivia}
Let $c$ and $c_1$ be axes of $\gamma$. Then
\begin{enumerate}
 \item\label{i-busemann-shift}
$b_c(\gamma x)=b_c(x)-\min d_\gamma$ for all $x\in M$.
 \item\label{i-busemann-decay}
$b_c$ decays at unit rate along $c_1$,
that is $b_c(c_1(t+t_1)) = b_c(c_1(t))-t_1$ for all $t,t_1\in\R$.
\end{enumerate}
\end{lemma}

\begin{proof}
Let $L=\min d_\gamma$. Then $b_c(\gamma x)=b_{\gamma^{-1}c}(x)=b_c(x)-L$
since $\gamma^{-1}c(t)=c(t-L)$ by Lemma~\ref{l-axes-trivia}(\ref{i-shift}).
This proves the first assertion. It implies that
$$
 b_c(c_1(t+mL)) = b_c(\gamma^m c_1(t)) = b_c(c_1(t)) - mL
$$
for all $m\in\Z$. Since $b_c$ is 1-Lipschitz, the second assertion follows.
\end{proof}

\section{Busemann functions of axes of isometries}\label{s:busemann}

The goal of this section in to prove the following.

\begin{prop}
\label{p-const-difference}
Let $\g\in\Gamma\setminus\{e\}$, and let $c$ and $c_1$ be axes of $\gamma$.
Then $b_c-b_{c_1}$ is constant on~$M$.
\end{prop}

This proposition was proved in the analytic case
in \cite[Proposition 2.5]{CS}.
The analyticity assumption
is used  in \cite{CS} to ensure some regularity
of the set $A_\gamma$, namely it implies
that this set is {\it locally rectifiably path connected} (see~\cite{CS} for the definition).
This property and local analysis of Busemann
functions on $A_\gamma$ yields the result.
We do not know if $A_\gamma$ is locally
rectifiably path connected if the metric is only $C^\infty$.
We work around this issue by more delicate analysis
of the behavior of Busemann functions in a neighborhood
of $A_\gamma$ (Lemma~\ref{l-busemann-estimate}).

For a geodesic $c$ in $M$, denote by $b_c^-$
the Busemann function of the reverse geodesic
$t\mapsto c(-t)$
and let $b_c^0=b_c+b_c^-$.
The definition of a Busemann function
and the triangle inequality imply that
$b_c^0(x)\ge 0$ for all $x\in M$.

\begin{lemma}\label{l-busemann-estimate}
Let $c$ be an axis of $\gamma$ and $L=\min d_\gamma$.
Define $f(x)=d_\gamma(x)-L$ for $x\in M$.
Then there exists $C>0$ such that
\begin{equation}
\label{e-busemann-estimate}
 |b_c^0(x)-b_c^0(y)| \le C\big( d(x,y)^2 + f(x) +f(y) \big)
\end{equation}
for all $x,y\in M$.
\end{lemma}

\begin{proof}
Note that $|b_c^0(x)-b_c^0(y)|\le 2d(x,y)$ since Busemann functions
are 1-Lipschitz.
Therefore we may assume that $d(x,y)$ is small, more precisely $d(x,y)\le L/2$.
Indeed, if $d(x,y)>L/2$ then
$|b_c^0(x)-b_c^0(y)|\le 2d(x,y)\le 4L^{-1}d(x,y)^2$,
so \eqref{e-busemann-estimate} is satisfied for any $C\ge 4L^{-1}$.
We may also assume that $f(x)\le 1$ and $f(y)\le 1$, otherwise
$$
|b_c^0(x)-b_c^0(y)|\le 2d(x,y) \le L \le L(f(x)+f(y)),
$$
so \eqref{e-busemann-estimate} is satisfied for any $C\ge L$.

Let $x_+=\gamma x$ and $x_-=\gamma^{-1}x$.
Then
$$
 d(x,x_+)=d(x,x_-) = d_\gamma(x) = L+f(x) .
$$
By Lemma~\ref{l-busemann-trivia}(\ref{i-busemann-shift}),
we have $b_c(x)=b_c(x_+)+L$ and
$b_c^-(x)=b_c^-(x_-)+L$.
Since Busemann functions are 1-Lipschitz,
it follows that
$$
 b_c(y)-b_c(x) = b_c(y) - b_c(x_+) - L \le d(y,x_+)-L
$$
and
$$
 b_c^-(y)-b_c^-(x) = b_c(y) - b_c^-(x_-) - L \le d(y,x_-)-L .
$$
Summing these two inequalities yields
\begin{equation}
\label{e-busemann1}
 b_c^0(y)-b_c^0(x) \le d(y,x_+)+d(y,x_-)-2L .
\end{equation}
Our plan is to estimate the right-hand side of this inequality.

For $p,q\in M$, denote by $\overrightarrow{pq}$
the initial velocity vector of the unique geodesic
connecting $p$ to $q$. That is, $\overrightarrow{pq}$
is the unit vector in $T_pM$ positively proportional
to $\exp_p^{-1}(q)$.
Let $v_+=\overrightarrow{xx_+}$, $v_-=\overrightarrow{xx_-}$
and $z\in M$ be any point such that $d(x,z)\le L/2$.
Then, by the first variation formula,
$$
\begin{aligned}
 d(z,x_\pm)&\le d(x,x_\pm) -\langle \overrightarrow{xz},v_\pm\rangle\cdot d(x,z) + C_1d(x,z)^2 \\
 &= L + f(x) -\langle \overrightarrow{xz},v_\pm\rangle\cdot d(x,z) + C_1d(x,z)^2
\end{aligned}
$$
and therefore
\begin{equation}
\label{e-1stvariation}
 d(z,x_+)+d(z,x_-)
\le 2L + 2f(x) -\langle \overrightarrow{xz},v_++v_-\rangle\cdot d(x,z) + 2C_1d(x,z)^2
\end{equation}
where $C_1>0$ is a constant depending only on $M$ and $L$.
(This constant is just an upper bound for the
second derivative of the distance to $x_\pm$
on the geodesic segment $[x,z]$. It is uniform in $x$
because $M$ admits a co-compact action by isometries
and $L\le d(x,x_\pm)\le L+1$.)

For $z=y$, \eqref{e-busemann1} and \eqref{e-1stvariation} imply that
\begin{equation}
\label{e-busemann2}
\begin{aligned}
 b_c^0(y)-b_c^0(x)
 &\le 2f(x) - \langle\overrightarrow{xy}, v_++v_-\rangle\cdot d(x,y) +2C_1 d(x,y)^2 \\
 &\le 2f(x) + \sigma\cdot d(x,y)+2C_1 d(x,y)^2 .
\end{aligned}
\end{equation}
where $\sigma={|v_++v_-|}$.

It remains to estimate $\sigma$.
Consider a point $z=\exp_x(t_0v)$ where
$v\in T_xM$ is the unit vector positively proportional
to $v_++v_-$ (or any unit vector if $v_++v_-=0$)
and $t_0=\sqrt{f(x)/C_1}$.
We may assume that $C_1$ is sufficiently large so that
$t_0\le L/2$ (recall that $f(x)\le 1$).
For the so defined $z$, \eqref{e-1stvariation} takes the form
$$
 d(z,x_+)+d(z,x_-) \le  2L+2f(x)-\sigma t_0 + 2C_1t_0^2 .
$$
On the other hand,
$$
d(z,x_+)+d(z,x_-)\ge d(x_+,x_-)=d_{\gamma^2}(x_-)\ge\min d_{\gamma^2} =2L
$$
by the triangle inequality and 
Lemma~\ref{l-axes-trivia}(\ref{i-mindm}) (recall that $x_+=\gamma^2x_-$).
Hence
$$
 0\le 2f(x)-\sigma t_0 + 2C_1t_0^2 = -\sigma\sqrt{f(x)/C_1} + 4f(x),
$$
or, equivalently, $\sigma\le 4\sqrt{C_1f(x)}$. Therefore
$$
 \sigma\cdot d(x,y) \le 4\sqrt{C_1f(x) d(x,y)^2} \le 2f(x) + 2C_1 d(x,y)^2 .
$$
Plugging this into \eqref{e-busemann2} yields
$$
 b_c^0(y)-b_c^0(x) \le 4f(x) + 4C_1 d(x,y)^2 
$$
and, by switching the roles of $x$ and $y$,
$$
 b_c^0(x)-b_c^0(y) \le 4f(y) + 4C_1 d(x,y)^2 .
$$
Therefore \eqref{e-busemann-estimate}
is satisfied for $C\ge\max\{4C_1,4\}$.
\end{proof}

The following lemma is purely analytic
and local. It does not use the assumption that
$M$ is co-compact and free of conjugate points.

\begin{lemma}\label{l-analytic}
Let $M$ be any Riemannian manifold, $f\in C^\infty(M)$
and $g\co M\to\R$ a continuous function such that
\begin{equation}
\label{e-gdiff}
 |g(x)-g(y)| \le C \big(d(x,y)^2 +|f(x)|+|f(y)|\big)
\end{equation}
for some constant $C>0$ and all $x,y\in M$.
Let $X=f^{-1}(0)$.
Then the set $g(X)\subset\R$ has zero Lebesgue measure.
\end{lemma}

\begin{proof}
We argue by induction in $n=\dim M$.
The case $n=0$ is trivial. Assume that $n\ge 1$
and the lemma holds for all $(n-1)$-dimensional manifolds.

Since $M$ can be covered by countably many coordinate neighborhoods,
we may assume that $M$ is an open ball in $\R^n$
and $f$ extends to a smooth function on a compact set.
Let $Z$ be the set of points in $X$ where all derivatives of $f$ vanish,
and let $Y=X\setminus Z$.

The set $Y$ is contained in a countable union of $(n-1)$-dimensional
smooth submanifolds. Indeed, for each multi-index $I=(i_1,\dots,i_k)\in\{1,\dots,n\}^k$,
$k\ge 0$, consider the partial derivative $f_I=\frac{\pd^kf}{\pd x_1\dots\pd x_k}$.
Let $\Sigma_I$ be the set of points in $M$ where
the function $f_I$ vanish but its first derivative does not.
Since $f_I\in C^\infty$, $\Sigma_I$ is a smooth $(n-1)$-dimensional
submanifold, and $Y$ is contained in the union of all such submanifolds.

Applying the induction hypothesis to $\Sigma_I$ in place of $M$
yields that $g(X\cap\Sigma_I)$ is a set of measure zero for every multi-index $I$.
Therefore $g(Y)$ has measure zero.

It remains to handle the set $g(Z)$.
Since $f$ and all its derivatives up to the order $2n$
vanish on $Z$ and are bounded on $M$,
there is a constant $C_1>0$ such that
\begin{equation}
\label{e-high-order-vanish}
 |f(x)|=|f(x)-f(z)| \le C_1 |x-z|^{2n}
\end{equation}
for all $x\in M$ and $z\in Z$.
We are going to show that
\begin{equation}
\label{e-power-of-distance}
 |g(z)-g(z')| \le C_2 |z-z'|^{n+1}
\end{equation}
for some constant $C_2>0$ and all $z,z'\in Z$ such that $|z-z'|\le 1$.
Let $\delta=|z-z'|$ and pick a positive integer $N$
such that $\delta^{1-n}\le N\le 2\delta^{1-n}$.
Divide the segment $[z_1,z_2]$ into $N$ segments
$[x_{i-1},x_i]$, $i=1,\dots,N$, of equal lengths. Note that
$$
 |x_i-x_{i-1}| = \delta/N \le \delta^n
$$
by the choice of $N$, and
$$
 |f(x_i)|+|f(x_{i-1})| \le C_1|x_i-z|^{2n}+C_1|x_{i-1}-z|^{2n} \le 2C_1\delta^{2n}
$$
by \eqref{e-high-order-vanish}.
Substituting $x_i$ and $x_{i-1}$ for $x$ and~$y$
in \eqref{e-gdiff} yields that
$$
 |g(x_i)-g(x_{i-1})| \le C ( |x_i-x_{i-1}|^2 + |f(x_i)| + |f(x_{i-1})| )
 \le C(1+2C_1)\delta^{2n} .
$$
Hence
$$
 |g(z)-g(z')| = \bigl| \sum g(x_i)-g(x_{i-1}) \bigr|
 \le  C(1+2C_1)\delta^{2n} N \le 2C(1+2C_1)\delta^{n+1} .
$$
Thus \eqref{e-power-of-distance} holds for $C_2=2C(1+2C_1)$.

The inequality \eqref{e-power-of-distance} implies that
$$
 \dim_H(g(Z))\le\frac{\dim_H(Z)}{n+1}\le\frac{\dim_H(M)}{n+1}= \frac{n}{n+1}<1
$$
where $\dim_H$ denotes the Hausdorff dimension.
Hence the one-dimensional measure of $g(Z)$ is zero
and the lemma follows.
\end{proof}

\begin{cor}
\label{c-b0}
Let $c$ be an axis of $\gamma$.
Then $b_c^0=0$ on $A_\gamma$.
\end{cor}

\begin{proof}
Let $L=\min d_\gamma$, then
Lemma~\ref{l-busemann-estimate} asserts that the functions $f=d_\gamma-L$
and $g=b_c^0$ satisfy the assumptions of Lemma~\ref{l-analytic}.
Since $A_\gamma=f^{-1}(0)$, Lemma~\ref{l-analytic} implies that
$b_c^0(A_\gamma)$ is a set of measure zero in $\R$.
By Lemma~\ref{l-connected}, $A_\gamma$ and hence $b_c^0(A_\gamma)$ are connected.
Since the set $b_c^0(A_\gamma)$ is connected and has measure zero,
it is a single point on the real line. This means that $b_c^0$ is constant on $A_\gamma$.
Since $c(0)\in A_\gamma$ and $b_c^0(c(0))=0$, this constant is zero.
\end{proof}

\begin{proof}[Proof of Proposition~\ref{p-const-difference}]
Let $c$ and $c_1$ be axes of $\gamma$.
Changing the origin of $c_1$ we may assume that $b_c(c_1(0))=0$.
Then we have to prove that $b_c=b_{c_1}$.

First we show that the relation $b_c(c_1(0))=0$ is symmetric,
i.e.\ it implies that $b_{c_1}(c(0))=0$.
Since $c_1(0)\in A_\gamma$, we have $b_c^0(c_1(0))=0$ by Corollary~\ref{c-b0},
hence $b_c^-(c_1(0))=0$. This means that
$$
 d(c_1(0),c(-t)) = t +\ep(t)
$$
where $\ep(t)\to 0$ as $t\to+\infty$.
Since $b_{c_1}(c(-t)) \le d(c(-t),c_1(0))$,
it follows that $b_{c_1}(c(-t)) \le t+\ep(t)$.
This and Lemma~\ref{l-busemann-trivia}(\ref{i-busemann-decay}) imply that
$$
 b_{c_1}(c(0)) = b_{c_1}(c(-t)) - t \le \ep(t)
$$
for all $t>0$.
Therefore $b_{c_1}(c(0))\le 0$. Similarly, $b_{c_1}^-(c(0))\le 0$.
Both inequalities must turn to equalities because $b_{c_1}+b_{c_1}^-\ge 0$.
Thus $b_{c_1}(c(0))=0$ as claimed.

Let $x\in M$ and $\ep>0$. By the definition
of Busemann function, there exists $t_1\in\R$ such that
$$
 d(x,c_1(t_1)) < b_{c_1}(x) + t_1 + \ep .
$$
Since $b_c(c_1(0))=0$, Lemma~\ref{l-busemann-trivia}(\ref{i-busemann-decay})
implies that $b_c(c_1(t_1))=-t_1$,
hence
$$
 d(c_1(t_1),c(t_0)) < t_0-t_1+\ep 
$$
for a sufficiently large $t_0\in\R$. Therefore
$$
 d(x,c(t_0)) \le d(x,c_1(t_1))+d(c_1(t_1),c(t_0))<b_{c_1}(x) + t_0 + 2\ep
$$
and hence $b_c(x)\le b_{c_1}(x)$ since $\ep$ is arbitrary.

Swapping $c$ and $c_1$ in this argument yields that $b_{c_1}(x)\le b_c(x)$.
Thus $b_c=b_{c_1}$ and the proposition follows.
\end{proof}

\section{Virtual splitting of centralizer}
\label{sec-proof-main}

\begin{lemma}\label{l-busemann-shift-by-commutator}
Let $\alpha\in Z(\gamma)$ and $c$ be an axis of $\gamma$.
Then $b_c(\alpha x)-b_c(x)$ does not depend on $x\in M$.
\end{lemma}

\begin{proof}
This is \cite[Corollary 2.6]{CS} without the analyticity assumption.
Observe that $\alpha^{-1} c$ is an axis of $\gamma$, indeed,
if $L=\min d_\gamma$ then
$$
 \gamma\alpha^{-1} c(t) = \alpha^{-1}\gamma c(t) = \alpha^{-1} c(t+L) .
$$
Therefore $b_{\alpha^{-1}c}-b_c$ is constant by Proposition~\ref{p-const-difference}.
Since $b_{\alpha^{-1}c}(x)=b_c(\alpha x)$, the lemma follows.
\end{proof}

\begin{cor}\label{c-homomorphism}
There exist a homomorphism $h\co Z(\gamma)\to\R$
such that $h(\gamma)\ne 0$.
\end{cor}

\begin{proof}
Let $c$ be an axis of $\gamma$.
For $\alpha\in Z(\gamma)$, define
$$
 h(\alpha) = b_c(\alpha x)-b_c(x)
$$
where $x\in M$ is an arbitrary point.
By Lemma~\ref{l-busemann-shift-by-commutator},
$h(\alpha)$ is well defined.
For $\alpha,\beta\in Z(\gamma)$ we have
$$
 h(\alpha\beta) =  b_c(\alpha\beta x)-b_c(x)
 = b_c(\alpha\beta x)-b_c(\beta x)+b_c(\beta x)-b_c(x)
 = h(\alpha)+h(\beta).
$$
Thus $h$ is a homomorphism.
By Lemma~\ref{l-busemann-trivia}(\ref{i-busemann-shift}),
$h(\gamma)=-\min d_\gamma\ne 0$.
\end{proof}

Recall that $H_1(Z(\gamma))=Z(\gamma)/[Z(\gamma),Z(\gamma)]$.
We denote by $\pi$ the projection from $Z(\gamma)$
to $H_1(Z(\gamma))$.

\begin{cor}\label{c-infinite-order}
$\pi(\gamma)$ is non-torsion (i.e. has infinite order) in $H_1(Z(\gamma))$.
\end{cor}

\begin{proof}
Since $\R$ is commutative, the homomorphism $h$
from Corollary~\ref{c-homomorphism} factors
as $h=\bar h\circ\pi$ where $\bar h$ is a homomorphism
from $H_1(Z(\gamma))$ to $\R$.
Since $\bar h(\pi(\gamma))\ne 0$,
$\pi(\gamma)$ has infinite order.
\end{proof}

\begin{proof}[Proof of Theorem~\ref{t-main}]
Recall that by Remark~\ref{centralizer-fg} $Z(\gamma)$ is finitely generated and hence so is  $H_1(Z(\gamma))$.
By Corollary~\ref{c-infinite-order}
and the classification of finitely generated abelian groups,
$\pi(\gamma)$ belongs to a finite index subgroup $H<H_1(Z(\gamma))$
which is isomorphic to $\Z^k$ so that $\pi(\gamma)$ is mapped
to $(1,0,\dots,0)\in\Z^k$ by this isomorphism.
Let $p\co H\cong\Z^k\to\Z$ be first coordinate projection.
Then $G:=\pi^{-1}(H)$ is a finite index subgroup of $Z(\gamma)$
and $\phi=p\circ\pi\co G\to\Z$ is a homomorphism sending $\gamma$
to $1\in\Z$. The existence of such a homomorphism and the fact
that $\gamma$ belongs to the center of $G$ imply
that $G\cong\Z\times G'$
where the $\Z$ factor corresponds to
the subgroup generated by $\gamma$ and $G'=\ker \phi$.
\end{proof}

\section{Examples}\label{s:examples}
In this section we will give some new examples of manifolds that can
be shown not to admit metrics without conjugate points using Theorem~\ref{t-main}.

\begin{example}\label{ex-center}
Let $S^2_g$ be a closed orientable surface of genus $g>1$ equipped with any Riemannian metric.
Let $\bar M^3=T^1S^2_g$ be the unit tangent bundle to $S^2_g$.
Then $\bar M^3$ does not admit a metric without conjugate points.

Indeed, since the Euler class of $S^2_g$ is not equal to zero, $\pi_1(\bar M)$
is a nontrivial central extension $1\to\Z\to \pi_1(\bar M)\to \pi_1(S^2_g)\to 1$.
By Theorem~\ref{t-main} it does not admit a metric without conjugate points.
On the other hand, $\pi_1(\bar M)$ is semi-hyperbolic~\cite{AB95} and therefore it satisfies
the previously known restrictions on fundamental groups
of manifolds without conjugate points proved in~\cite{CS},
namely, every abelian subgroup of $\pi_1(\bar M)$ is straight
and every solvable subgroup of $\pi_1(\bar M)$ is virtually abelian.
\end{example}

Recall that the key step in the proof of Theorem ~\ref{t-main} is Proposition~\ref{p-const-difference}.
It  was proved in ~\cite{CS} under the assumption  that $A_\gamma$ is locally rectifiably path connected.
As was discussed earlier this can be guaranteed if the metric is real analytic.
Another condition that obviously ensures it is if $\gamma$ belongs to the center of  $\pi_1(\bar M)$.
Indeed, if $\g\in Z(\pi_1(\bar M))$ then $d_\g$ descends to a well-defined function on $M$
and hence it is bounded above and attains its maximum.
By Lemma~\ref{l-axes-trivia}(\ref{l-axes-crit}) this means that $\max d_\g=\min d_\g$,
i.e.  $d_\g$ is constant.
That means that $A_\g=M$ which is obviously rectifiably path connected.
Therefore one can prove that $\bar M$ admits no metric without conjugate points
without using Proposition~\ref{p-const-difference}.

Next we will give an example which has no finite index subgroups with nontrivial center and requires
the full strength of Proposition~\ref{p-const-difference} and Theorem~\ref{t-main}.

\begin{example}
First let us recall some basic facts about diffeomorphism groups of surfaces. 
Let $F$ be a closed orientable surface of genus $>1$.
Then  the mapping class group $\Mod(F)$ is defined as $\pi_0(\Diff(F))$.
It can be alternatively described as $\Diff(F)/\Diff_0(F))$ where
$\Diff_0(F)$ is the identity component of $\Diff(F)$.
By a classical result of Dehn and Nielsen $\Mod(F)$ is isomorphic
to the group $\Out(\pi_1(F))=\Aut(\pi_1(F))/\Inn(\pi_1(F))$
of outer automorphisms of $\pi_1(F)$.
For a once-punctured surface $F\setminus\{pt\}$, we have
$\Mod(F\setminus\{pt\})\cong \Aut(\pi_1(F))$.

It is well known that $\Diff_0(F)$ is contractible~\cite{EE}.

\begin{lemma}\label{l-injective-rho}
Let $S_g^2$ be a closed orientable surface of genus $g>1$,
$\bar M^3=T^1S_g$ and $F=S_g\# S_g$.
Then there exists a monomorphism $\rho\co\pi_1(\bar M)\to \Mod (F)$
which admits a lift to a monomorphism $\bar\rho\co\pi_1(\bar M)\to \Mod(F\setminus\{pt\})$.
\end{lemma}

\begin{proof}
Pick $p\in S_g$ and $v\in T^1_pS_g$. Let $\Diff(S_g,v)$ be the subgroup of $\Diff (S_g)$ fixing $v$.
Consider the fibration  $\Diff(S_g,v)\to\Diff(S_g)\overset{ev}{\to}\bar M^3$ where the last map is just
the evaluation map on $v$. Look at the long exact homotopy sequence of this fibration
 \[
 \ldots\to\pi_1(\Diff (S_g))\to \pi_1(\bar M^3)\overset{i}{\to} \pi_0 (\Diff(S_g,v))\to \pi_0 (\Diff(S_g))
 \]
 
As mentioned earlier $\Diff_0 (S_g)$ is contractible; therefore $\pi_1(\Diff (S_g))=1$ and hence
$\pi_1(\bar M^3)$ injects as a subgroup into $\pi_0 (\Diff(S_g,v))$.
Next choose an $\ep>0$ much smaller than the injectivity radius of $S_g$
and let $S^+=S_g\setminus B_\ep(p)$.
Then $\Diff(S_g,v)$ is naturally isomorphic to $\Diff(S^+,\partial S^+)$,
the group of diffeomorphisms of $S^+$ fixing the boundary $\partial S^+$ pointwise.
Let $F=S^+\cup_{\partial S^+}S^-$ where $S^-$ is another copy of $S^+$
(i.e. $F$  is the double of $S^+$ along its boundary).
Then $\Diff(S^+,\partial S^+)$  naturally embeds into $\Diff (F)$
by extending the diffeomorphisms to $S^-$ by identity.
This induces a map $j\co \pi_0(\Diff(S^+,\partial S^+))\to \pi_0 (\Diff (F))$ which is also injective.
Then $\rho=j\circ i\co \pi_1(\bar M)\to \Mod (F)$ is injective.

The above construction (originally due to Mess~\cite{Mess90})
is borrowed from \cite[Section 4.1]{KaLe},
see also ~\cite{Birman74}.

Next observe that if instead of gluing $S^-$ to $S^+$ we glue
$S^-\setminus \{pt\}$ the same argument gives an embedding
$\pi_1(\bar M)\to \Mod (F\setminus \{pt\})$
so that the above map $\rho\co  \pi_1(\bar M)\to \Out(\pi_1(F))$ lifts to a map 
$\bar\rho \co \pi_1(\bar M)\to \Mod (F\setminus \{pt\})$ and both of these are injections.
\end{proof}

Let $B\Diff(F)=E/\Diff(F)$ be the classifying space of $\Diff(F)$
(here $E$ is contractible).
Then the projection map  $p\co E\to B\Diff(F)$ factors through the map
$$
E\to E/\Diff_0(F)\overset{\bar p}{\to} E/\Diff(F)= B\Diff(F)
$$
where $\bar p$ is just the quotient by $\Diff(F)/\Diff_0(F))=\Mod(F)$.
Since $\Diff_0(F)$ is contractible then so is $E/\Diff_0(F)$ and therefore
$B\Diff(F)$ can be identified with $B\Mod(F)$ which we will do from now on.
Therefore, for any CW complex $B$ any map $f\co B\to B\Mod(F)$
(which up to homotopy is determined by the map on the fundamental groups)
gives rise to a pullback bundle $\Diff(F)\to X\to B$ and
its associated bundle $F\to N\to B$
coming from the natural action of $\Diff(F)$ on $F$.

Now consider the map $\bar M \to B\Mod (F)$ corresponding to the monomorphism 
$\rho\co\pi_1(\bar M)\to\Mod(F)$ from Lemma~\ref{l-injective-rho}
and pull back the universal $F$ bundle.
This gives us a $5$-manifold $N$ fibering over $\bar M$ with fiber $F$.
Clearly $N$ is aspherical and  its fundamental groups fits into a short
exact sequence
$$
1\to \pi_1(F)\to\pi_1(N)\to\pi_1(\bar M)\to 1 .
$$
Since $\rho$ can be lifted to $\bar\rho\co\pi_1(\bar M)\to\Mod(F\setminus\{pt\})\cong\Aut(\pi_1(F))$,
this exact sequence  splits, i.e.\
$\pi_1(N)$ is isomorphic to the semidirect product $\pi_1(F)\rtimes_{\bar\rho}\pi_1(\bar M)$.

Since $\bar \rho$ is injective,  $\pi_1(N)$ has trivial center.
On the other hand it contains $\pi_1(\bar M)$ as a subgroup and therefore
it has an element whose centralizer does not virtually split.
Therefore, $N$ does not admit a metric without conjugate points by Theorem~\ref{t-main}.
Lastly note that since $\rho$ is injective, the whole group $\pi_1(N)$ injects into $\Aut(\pi_1(F))$
via the conjugation action. Also recall that $\Aut(\pi_1(F))\cong \Mod (F\setminus \{pt\})$.
It is well known  that  mapping class groups of hyperbolizable surfaces satisfy the property
that all their abelian subgroups are straight (see~\cite{FLM2001} and \cite[Lemma 8.7]{Ivanov92})
and all their solvable subgroups are virtually abelian ~\cite{BLM83}.
Of course the same is therefore  true for all their subgroups and hence it is true for $\pi_1(N)$.
\end{example}

\section{Solvability of word and conjugacy problems}\label{s:solv}
It is very well known that fundamental groups of manifolds of nonpositive curvature have solvable word and conjugacy problems. We observe that the same is true for fundamental groups of closed manifolds without conjugate points.

\begin{theorem}\label{t-solv-word}
Let $\bar M$ be a closed manifold with a  Riemannian metric without conjugate points. Then $\pi_1(\bar M)$ has solvable word and conjugacy problems.

\end{theorem}

\begin{proof}
The proof is a straightforward combination of known results.
The key is the following Lemma.

\begin{lemma}\label{l-fil-length}
Let $c_0,c_1\co S^1\to\bar M$ be two freely homotopic rectifiable loops with length $L(c_i)<C$ for $i=0,1$.
Then there exists a free homotopy $F\co S^1\times [0,1]\to\bar M$ from $c_0$ to $c_1$ through loops $c_t$
such that $L(c_t)<C$ for any $t\in [0,1]$.
\end{lemma}

\begin{proof}
Lifting the homotopy to the universal cover $M$ we get a homotopy
between two curves $\tilde c_0,\tilde c_1\co[0,1]\to M$.
Since the metric has no conjugate points,
it is easy to connect each $\tilde c_i$, $i=1,2$, to the (unique) geodesic
segment $[\tilde c_i(0),\tilde c_i(1)]$ by a homotopy
fixing the endpoints and not increasing lengths. (For example, consider
a family of curves $\{\tilde c_{i,t}\}, t\in[0,1]$, where $\tilde c_{i,t}$
is the concatenation of the geodesic segment $[\tilde c_i(0),\tilde c_i(t)]$
and the arc $\tilde c_i|_{[t,1]}$ of $\tilde c_i$.)
Thus we may assume that $\tilde c_0$ and $\tilde c_1$ are geodesics.

Let $\gamma$ be the element of the deck transformation group such that
$\gamma(\tilde c_0(0))=\tilde c_0(1)$,
then $\gamma(\tilde c_1(0))=\tilde c_1(1)$.
Thus by our assumption both $\tilde c_0(0)$ and $\tilde c_1(0)$ belong to the set $\{d_\g<C\}$.
By Lemma~\ref{l-axes-trivia}(\ref{l-axes-crit}) and the proof of Lemma~\ref{l-connected}
the set $\{d_\g<C\}$ is path connected
and therefore we can find a path $\alpha\co [0,1]\to \{d_\g< C\}$ such that
$\alpha(0)=\tilde c_0(0)$ and $\alpha(1)=\tilde c_1(0)$.
Connect $\alpha(s)$ to $\gamma(\alpha(s))$ by shortest geodesics for all $s\in [0,1]$.
They are unique and  vary continuously since $M$ has no conjugate points.
Projecting this family of geodesics to $\bar M$ produces a homotopy with desired properties.
\end{proof}

We are indebted to Ilya Kapovich for providing the following argument that the property
provided by Lemma~\ref{l-fil-length} implies  solvability of the word problem in  $\pi_1(\bar M)$.
Since the argument is apparently well-known we only give a brief outline.
 
 It is well-known that to prove solvability of the word problem in a finitely presented group it is enough to show that its Dehn function has a recursive upper bound (see for example ~\cite[Theorem 1.5.8, p. 104]{BRS}).
 
 We claim that Lemma~\ref{l-fil-length} yields an exponential upper bound on the Dehn function of $\pi_1(\bar M)$.
 
Fix a finite presentation $\langle F|R\rangle$ of $\Gamma=\pi_1(\bar M,p)$ and
a $K$-quasi-isometry between $M$ and the Caley graph of $\Gamma$ with
respect to that presentation sending $p$ to $e$.
 
Let $F\co S^1\times [0,1]\to\bar M$ be a  homotopy between a constant loop $c_0$ at $p$ and a  loop $c_1$.
Let $C=2L(c_1)$.  Since we are only interested in large scale estimates we can assume that $C\ge 2\diam \bar M$.
By Lemma~\ref{l-fil-length} we can change $F$ to a free homotopy such  that $L(c_t)\le C$ for any $t$.  
 
Consider a sufficiently  fine subdivision of $[0,1]$ given by $0=t_0<t_1<\ldots<t_N=1$.
For every $j=0,\ldots N$ connect $c_{t_j}(0)$ to the fixed base point $p=c_0(0)$ by a shortest geodesic.
For every $j$ this produces  a nullhomotopic loop  $c_j'$ based at $p$ of length $\le C+2\diam(\bar M)$.
Look at the corresponding curves in the  Caley graph and the words they represent
in the presentation $\langle F|R\rangle$. They all have lengths $\le K(C+2\diam(\bar M))\le 2KC$.
The ball of radius $R$ in the Caley graph has at most $\exp(C_2 \cdot R)$ vertices
for some explicit constant $C_2$ depending on $|F|$.
Therefore the loops  $c_j'$ produce at most $\exp(2C_2KC)$ distinct words in the group.
Therefore, if the original homotopy was ``too long'' we can cut out the parts between
repeating words and reduce its length. This produces a homotopy in the Caley graph
of area $\le \exp(C_3\cdot C)$ for some easily computable constant
$C_3=C_3(C_2,K, |F|+|R|)$. Therefore the Dehn function grows at most exponentially.

Lastly, by ~\cite[section 1.11, p. 446]{HB99} if $\pi_1(\bar M)$ has a solvable word problem
and satisfies the conclusion of Lemma~\ref{l-fil-length}, then  it also has solvable conjugacy problem.
\end{proof}

\begin{remark}
One can also prove the solvability of the word problem in $\Gamma$ as follows.
Since $\bar M$ is compact we have that $|sec(M)|\le K$ for some $K>0$ which trivially implies
that for any $p\in M$ the exponential map $\exp_p\co T_pM\to M$ is $\exp(C_1R)$ Lipschitz
on the $R$-ball $B_R(0)\subset T_pM$ for some $C_1>0$ and any $R>0$.
Using that geodesics in $M$ are globally distance minimizing it is not hard to show that $\exp_p^{-1}$
is also $\exp(C_2R)$ Lipschitz which implies an exponential bound on the isoperimetric
function in $M$ and hence the Dehn function of $\Gamma$ also grows at most exponentially. 

It is well-known that Dehn functions of $CAT(0)$-groups grow at most quadratically.
It is reasonable to expect that the same should be true for fundamental groups
of closed  manifolds without conjugate points but we have been unable to prove it.
\end{remark}

\section{3-manifolds without focal points}\label{section-no-focal}
Recall that a Riemannian manifold $M$ is said to have no focal points
if every embedded shortest geodesic $c\co (a,b)\to M$ has no focal points when viewed as a submanifold of $M$.
The class of manifolds without focal points is contained in the class of manifolds without conjugate points
and contains  all non positively curved manifolds.
It is therefore natural to wonder if closed manifolds without focal points always admit
 nonpositively curved metrics.

The main result of this section is the following theorem which answers this question
affirmatively in dimension 3.

\begin{theorem}\label{3-no-focal}
Let $M$ be a closed $3$-manifold. Then $M$ admits a Riemannian metric without focal points
if and only if it admits a Riemannian metric of nonpositive sectional curvature.
\end{theorem}

One of the main technical tools in the proof is the following result due to O'Sullivan~\cite{OSul-76}.

\begin{theorem}[Flat torus theorem]\cite[Theorem 2]{OSul-76}\label{focal-flat-torus}
Let $M$ be a closed Riemannian manifold without focal points. Let $A\le \pi_1(M)$ be a solvable subgroup.
Then there exists a flat totally geodesically embedded space form $i\co F^k\hookrightarrow M$ such that
$i_*\co \pi_1(F^k)\to \pi_1(M)$ is injective and  $A$ is a finite index subgroup in $i_*(\pi_1(F^k))$.
\end{theorem}

We also need the following result from the same paper.

\begin{prop}\cite[Proposition 5]{OSul-76}\label{p-axes}
Let $\bar M$ be a closed manifold without focal points
and $\gamma$ a nontrivial element of the deck transformation group $\Gamma\cong\pi_1(\bar M)$.
Then
\begin{enumerate}
\item The union of axes $A_\gamma$ is convex and closed in the universal cover $M$.
\item $A_\gamma$ is isometric to $\R\times N$ where $N$ is a smooth submanifold of $\bar M$, possibly with boundary.
The action of $\gamma$ preserves this splitting and acts by translations by $\min d_\g$ on the $\R$-factors.
\end{enumerate}
\end{prop}

\begin{proof}[Proof of Theorem~\ref{3-no-focal}]
To simplify the exposition we will only give the proof for orientable manifolds.
The non-orientable case is treated similarly.
Alternatively we can  appeal to a result of M.~Kapovich and Leeb ~\cite[Corollary 2.5]{KaLe}
that given a finite cover $M_1\to M_2$ between closed 3-manifolds,
$M_1$ admits a metric of nonpositive sectional curvature if and only if $M_2$ does.

Let $\bar M$ be a closed orientable  3-manifold without focal points. Since $\bar M$ is aspherical it is prime.
By the geometrization  it admits a geometric decomposition. By \cite{CS} or by Theorem~\ref{focal-flat-torus},
none of the geometric pieces are sol or nil. Also, if $\bar M$ is Seifert fibered then by Theorem~\ref{t-main}
it is finitely covered by a product $S^1\times S_g$ where $S_g$ is a closed surface of genus $\ge 2$.
All such manifolds obviously admit nonpositively curved metrics and hence so does $\bar M$.
Thus we may assume that $\bar M$ is not Seifert fibered. If its geometric decomposition contains
at least one hyperbolic piece then $\bar M$ admits a metric of nonpositive curvature by \cite{leeb95}.
Thus we may assume that $\bar M$ is orientable, aspherical, not Seifert fibered and its geometric
decomposition has no nil or sol pieces. This means that $\bar M$ is a graph manifold
and all its geometric pieces are modelled on $\mathbb  H^2\times\R$.

Let us consider a single geometric piece $E$ of $\bar M$.
It is Seifert fibered over a hyperbolizable 2-manifold with boundary.
By Theorem~\ref{focal-flat-torus} we can assume that all the boundary tori of $E$
(which are incompressible in $\bar M$) are flat and totally geodesic in $\bar M$.
Let $E'\to E$ be a finite normal cover which is topologically a product $S^1\times \Sigma$
where $\Sigma$ is a compact 2-manifold with boundary.
Let $\gamma$ be a nontrivial element of $\pi_1(\bar M)$ corresponding to the $S^1$  factor of $E'$.
Let $\pi\co M\to\bar M$ be the universal cover of $\bar M$
and let $p\co \tilde E\to E$ be the universal cover of $E$.
Since the inclusion $E\hookrightarrow \bar M$ is $\pi_1$-injective we can think of $\tilde E$ as  a subset of $M$.

Clearly, the preimages of the boundary tori of $E$ in $\tilde E$ belong to $A_\gamma$.
This easily implies that the whole $\tilde E$ is contained in $A_\gamma$
by essentially the same argument as in Example~\ref{ex-center}. 
Indeed, since  $\pi_1(E')\subset Z(\gamma)$ in $\pi_1(\bar M)$,
$d_\g$ descends from $\tilde E$ to a well-defined function $d_\g'$ on $E'$
and since $E'$ is compact this function attains its maximum and minimum there.
The whole boundary of $E'$ belongs to the set of minima of $d_\gamma'$ and since $d_\gamma$
has no critical points outside $A_\g$ by Lemma~\ref{l-axes-trivia}(\ref{l-axes-crit})
we get that $d_\g'$ is constant on $E'$ and $d_\g$ is constant on $\tilde E$.

Therefore by Proposition~\ref{p-axes}(2) $\tilde E$ isometrically splits as a product.
More precisely, $\tilde E$ is isometric to $\R\times N^2$ where $N$
is the universal cover of $\Sigma$ equipped with some Riemannian metric with totally geodesic boundary.
The isometric action of $\pi_1(E')=\pi_1(S^1)\times \pi_1(\Sigma)$ on $\R\times N$
preserves this splitting, furthermore the action of the $\pi_1(S^1)$ factor
preserves the $\R$-fibers of $\R\times N$ and the action of the $\pi_1(\Sigma)$ factor
descends to the action on $N$ by deck transformations.
Thus the metric of $N$ descends to a well defined Riemannian metric $g_0$ on $\Sigma$.
Note that we {\it do not} claim that $E'$ itself is isometric to a product.

Let us apply the uniformization procedure to the double of $\Sigma$.
It produces a  metric  $g_1$ of constant negative curvature in the same conformal class as the original metric $g_0$.
Since the uniformization procedure commutes with isometries the boundary of $\Sigma$ remains totally geodesic.
Also, the isometric action (by deck transformations)
of $\pi_1(\Sigma)$ on $N$  with respect to $\tilde g_0$ remains
isometric 
with respect to $\tilde g_1$.
Therefore the action of $\pi_1(E)$ on $\tilde E$ is isometric with respect not only to the original
product metric $\R\times (N,\tilde g_0)$ but also with respect to the new metric $\R\times (N,\tilde g_1)$
which is nonpositively curved.
Notice that the conformal change of $g_0$ to $g_1$ may change the lengths of the boundary circles of $\Sigma$
but we can easily modify the metric $g_1$ to $g_2$ by rescaling and attaching tubular collars near
the boundary so that the $g_2$ is still nonpositively curved and all the boundary circles
have the same lengths as in $g_0$  and are still totally geodesic.
Moreover we can can easily arrange that $g_2$ is a product near the boundary.
This yields a metric  on $E$ which is nonpositively curved,
$\partial E$ is totally geodesic and isometric to the original boundary,
and the metric near $\partial E$ is flat.
Doing it on all geometric pieces of $\bar M$ separately and gluing the resulting
metrics together yields a metric of nonpositive curvature on $\bar M$.
\end{proof}

\section{Open problems}\label{s:open}

The main open question concerning manifolds without conjugate points is Question~\ref{q-nonpos} which we restate here.

\begin{question}
Does every closed Riemannian manifold without conjugate points admit a nonpositively curved metric?
The same can be asked about manifolds without focal points.
\end{question}

While we suspect that this is likely false in general,  it might be true in dimension~3
(cf.\ Theorem~\ref{3-no-focal}). The main problem is to understand which graph manifolds
admit metrics without conjugate points.
The corresponding question about nonpositively curved metrics is completely understood
by work of Buyalo and Kobelski ~\cite{BuKo2}.

The simplest test case to understand is the following

\begin{question}
Let $\Sigma=T^2\setminus D^2$. Let $\bar M_1=\bar M_2=\Sigma\times S^1$.
Let $\bar M=\bar M_1\cup_f\bar M_2$ where $f$ is a self diffeomorphism of the boundary torus $T^2=\partial \bar M_1$
whose action on $\Z^2=\pi_1(T^2)$ is given by a matrix $A\in SL(2,\Z)$
with $|{\operatorname{trace}(A)}|>2$.
Then it is easy to see that $\bar M$ does not admit a metric of nonpositive curvature
(say, by Theorem~\ref{focal-flat-torus} and Proposition~\ref{p-axes}).
\emph{Does $\bar M$ admit a metric without conjugate points?}
\end{question}

\begin{question}
One of the natural classes containing CAT(0)-groups is the class of semi-hyperbolic groups.
As was remarked in the introduction it is known that in semi-hyperbolic groups
all abelian subgroups are straight and all solvable subgroups are virtually abelian.
It is therefore natural to pose the following weaker version of Question~\ref{q-nonpos}.

\emph{Is the fundamental group of a closed manifold without conjugate points semi-hyperbolic?}

Note however, that Theorem~\ref{t-main} \emph{does not} hold for general semi-hyperbolic groups
with the  fundamental groups of unit tangent bundles to surfaces of genus $>1$ providing the simplest counterexamples.

Lastly, let us note that it is easy to see that if $\bar M$ has no focal points
then $\pi_1(\bar M)$ is semi-hyperbolic.
Indeed, no focal points condition implies that if $c_1(t),c_2(t)$ are  geodesics in $M$ with $c_1(0)=c_2(0)$
then $d(c_1(t),c_2(t))$ is monotone increasing for $t>0$ ~\cite[Proposition 2]{OSul-76}.
This trivially implies that the canonical bicombing of $M$ by shortest geodesics
satisfies the fellow traveller property and hence $\pi_1(\bar M)$ is semi-hyperbolic.
\end{question}

\begin{question}
Let $\bar M$ be a closed manifold without conjugate points and $\gamma\in\pi_1(\bar M)$.
Is it true that $Z(\gamma)$ is straight in $\pi_1(\bar M)$? This is easily seen to be true if $\bar M$ has no focal points because in this case $A_\gamma$  is convex in the universal cover $M$ and $A_\g/Z(\g)$ is compact.
More generally it is known to be true  for semi-hyperbolic groups~\cite{HB99}.
\end{question}


\end{document}